\documentclass{article}

\usepackage{graphicx} 
\usepackage{graphicx} 
\usepackage{graphicx,amssymb,amstext,amsmath,amsthm} 
\usepackage{url}
\usepackage[hidelinks,colorlinks=true,linkcolor=blue,citecolor=blue]{hyperref}
\usepackage{mathtools}
\usepackage[a4paper, margin = 3cm]{geometry}
\usepackage{xcolor}
\usepackage{float}
\usepackage{tikz-cd}
\usepackage[T1]{fontenc} 
\usepackage[utf8]{inputenc}
\usepackage{etoolbox}
\usepackage[backend=biber,style=alphabetic,doi=false,maxbibnames=4,isbn=false,url=false,giveninits=true]{biblatex} 

\newcommand{\q}[1]{``#1''}
\newcommand{\Sp}{{Spin$(7)$}} 
\theoremstyle{plain}
\newtheorem{thm}{Theorem}[section]
\newtheorem{lem}[thm]{Lemma}
\newtheorem{cor}[thm]{Corollary}
\newtheorem{prop}[thm]{Proposition}
\newtheorem*{thm*}{Theorem}
\theoremstyle{definition}
\newtheorem{defn}[thm]{Definition}
\newtheorem{remark}[thm]{Remark}
\newtheorem{example}[thm]{Example}

\title{Explicit solutions to the gradient flow of Spin(7)-structures}
\author{Joseph Duthie\footnote{Mathematical Institute, University of Oxford, Oxford OX2 6GG, United Kingdom\\
Email: joseph.duthie@maths.ox.ac.uk, ORCID: \href{https://orcid.org/0009-0008-2661-3573}{\color{black}{0009-0008-2661-3573}}}}
\date{21 November, 2025}

\begin{document}

\maketitle
\begin{abstract}
    We study the gradient flow of Spin($7$)-structures and construct the first explicit solutions, in the homogeneous setting. As an intermediate step, we obtain formulae expressing the Spin($7$)-torsion tensor and  gradient flow in terms of the Spin($7$)-torsion forms, which makes explicit computations more tractable. We use these formulae to find  explicit solutions to the gradient flow of Spin($7$)-structures, obtaining a shrinking soliton on $\mathrm{SU}(3)$ as well as another explicit solution on a certain $T^7$-bundle over $S^1$. We also find an explicit solution to the coupled Ricci-harmonic flow of Spin($7$)-structures. Finally, we consider the question of stability of solitons for the renormalised gradient flow, and show that the soliton on $\mathrm{SU}(3)$ admits stable directions, unstable directions, and zero modes.
\end{abstract}

{\hypersetup{linkcolor=black}
\tableofcontents
}

\section{Introduction}

In 1955, Berger classified which groups could possibly arise as holonomy groups of simply connected, irreducible, non-symmetric Riemannian manifolds \cite{Berger}. The question of which of these groups actually do arise as holonomy groups, namely all of them except Spin($9$), took longer to resolve. Two of the most interesting holonomy groups are the so-called \emph{exceptional holonomy groups}, $G_2$ and \Sp, which can only arise for Riemannian manifolds of dimensions $7$ and $8$ respectively. Here, we focus on the case of \Sp. Manifolds with holonomy Spin($7$) enjoy many noteworthy geometric properties. In particular, they are Ricci-flat, and more generally, manifolds with special holonomy currently provide the only non-trivial examples of compact Ricci-flat manifolds. The construction of manifolds with holonomy Spin($7$) (and $G_2$) proved to be a very challenging problem. The non-compact case was settled by Bryant in 1987 \cite{Bryant87}, and the first complete examples were constructed by Bryant-Salamon in 1989 \cite{BryantSalamon}. Finally, the compact case was resolved by Joyce in 1996 \cite{Joyce1996}.
There are now several more constructions, but they remain very rare.

More abundant than \Sp-manifolds are manifolds with \Sp-structures. A manifold $M$ admitting a \Sp-structure is equivalent to having a globally-defined $4$-form $\Phi$ of a particular type, and $M$ has holonomy contained in Spin($7$) if and only if $\Phi$ is parallel with respect to the Levi-Civita connection, in which case $\Phi$ is called \emph{torsion-free}.
The question of when an arbitrary \Sp-structure can be deformed into a torsion-free one is very difficult. At present, the only method in the compact case is due to Joyce \cite{Joyce1996} \cite{Joyce1999}, who developed a powerful perturbation technique for \Sp-structures with small torsion on certain manifolds obtained by resolving singularities on orbifolds of the form $T^8/\Gamma$ and Calabi-Yau-$4$-orbifolds.

Motivated by the successes of geometric flows in other areas of geometry, perhaps most notably in the resolution of the Poincaré conjecture by Perelman \cite{perelman2002entropyformularicciflow} \cite{perelman2003ricciflowsurgerythreemanifolds}, it is natural to approach the problem of finding manifolds with special holonomy via suitable geometric flows. In the $G_2$-setting, Bryant proposed the so-called \emph{Laplacian flow} \cite{BryantRemarks}, aiming to flow a closed $G_2$-structure towards a torsion-free one. The Laplacian flow has received significant attention, but there was until very recently a notable lack of study of flows of \Sp-structures. In \cite{Dwivedi24}, Dwivedi introduced the \emph{gradient flow of} \Sp-\emph{structures}, defined as the gradient flow of the $L^2$-norm of the torsion of a \Sp-structure, and whose critical points are precisely torsion-free \Sp-structures. 
\begin{defn}[\cite{Dwivedi24}]
    Let $(M,\Phi_0)$ be a compact $8$-manifold with initial Spin($7$)-structure $\Phi_0$. The gradient flow of \Sp-structures is the following initial value problem:

    \begin{equation}\label{GF}
            \begin{cases}
                 \frac{\partial}{\partial t} \Phi(t) = (-\operatorname{Ric} + 2(\mathcal{L}_{T_8}g) + T\star T - |T|^2g + 2 \text{ div}T)_t \diamond_t \Phi(t),\\
                 \Phi(0) = \Phi_0, \\
                 
            \end{cases} \tag{GF}
        \end{equation}
        where $T = T_{\Phi_t}$ is the torsion tensor, $T_8$ is a $1$-form component of $T$, $T \star T$ is a certain contraction of $T$ with itself,  $\operatorname{Ric}$ is the Ricci tensor, $|T|^2$ is the norm of $T$ taken with respect to the metric $g = g_{\Phi_t}$, $\operatorname{div}T_{jk} = g^{nm}\nabla_nT_{m;jk}$ and $\diamond_t$ is the infinitesimal action of $2$-tensors on $4$-forms, depending on $\Phi_t$. 
        
\end{defn}
All of the terms involved in this definition are those induced by the Spin($7$) structure $\Phi_t$, and will be defined in Section \ref{SectionPrelims}. We note that we use the summation convention throughout, and write $g^{ij}$ for the inverse of $g_{ij}$.

Currently, very little is known about this flow, with the only results being short-time existence and uniqueness and the non-existence of compact expanding solitons \cite{Dwivedi24}. A flow equivalent to this one in dimension $8$ but which exists in all dimensions has been studied from the spinorial point of view \cite{AmmannWeissWitt}, where some stability results have also been obtained \cite{schiemanowski2017stabilityspinorflow}.
However, there were no known explicit solutions and no knowledge of compact shrinking solitons. Part of the obstruction to obtaining such results was the difficulty in computing the terms of the flow equation above in explicit examples, and we address this issue in Section \ref{SectionReformulation}. 

\bigskip

\noindent\textbf{Outline of paper.}
In this paper, after recalling some preliminary information on \Sp-structures in Section \ref{SectionPrelims}, we express the torsion tensor $T$ in terms of the torsion forms $T^1_8$ and $T^5_{48}$, which are easier to compute. We also introduce the geometric flows of interest here. Then, in Section \ref{SectionReformulation}, we use a representation theoretic argument to express the terms of the gradient flow equation in terms of these torsion forms (Proposition \ref{PropReformulation}), resulting in the following.
\begin{thm*}[\ref{ThmTorsionForms}]
    Writing the gradient flow of \Sp-structures \ref{GF} as 
    \[
    \frac{\partial}{\partial t}\Phi_t = A_t \diamond_t \Phi,
    \]
    we can express $A$ as
    \begin{equation}
    \begin{split}
        A_t& = \left( -\frac{3}{16}\delta T^1_{8} - \frac{5}{32}| T^1_{8}|^2 +  \frac{1}{28}| T^5_{48}|^2\right)g_\Phi \\
        &+\mathsf{j} \left(  -\frac{1}{2}\delta(T^1_8 \wedge \Phi) -4\delta T^5_{48}   -\frac{3}{2}(T^1_8 \wedge *_\Phi T^5_{48}) -\frac{3}{8}*_\Phi (T^1_8 \wedge \Phi) \wedge T^1_8\right) \\
        & -\frac{1}{2}g_\Phi(\cdot \lrcorner *_\Phi T^5_{48},\cdot \lrcorner *_\Phi T^5_{48} ) \\
        &+\frac{7}{6}\pi_7(\mathrm{d}T^1_8) +\frac{7}{12}\pi_7(*_{\Phi}(T^1_8 \wedge T^5_{48})).
    \end{split}         
    \end{equation}
    where $\delta = -*_{\Phi_t}d*_{\Phi_t}$ is the codifferential of forms, $T^1_8$ and $T^5_{48}$ are the torsion forms, $\pi_7$ is the projection $\pi_7: \Omega^2 \rightarrow\Omega^2_7$ and $\mathsf{j}$ is a map from symmetric $2$-tensors to $4$-forms, depending on $\Phi_t$ and defined in Section \ref{SectionReformulation}.
\end{thm*}

This formulation makes explicit computations much more tractable. We also expect this formulation to be of independent interest, and work in progress is to use this to study dimensional reductions of the flow. We use this expression in Section \ref{SectionExamples} to construct the first explicit solutions to the gradient flow of \Sp-structures (Examples \ref{ExampleSU3Flow} and \ref{ExampleOtherSolution}), using homogeneous methods. In particular, we find that the natural bi-invariant \Sp-structure on $\operatorname{SU}(3)$ is a shrinking soliton.
We roughly summarise these results here, combining Examples \ref{ExampleSU3Flow}, \ref{ExampleOtherSolution} and Theorem \ref{ThmSolitonsExist}.

\begin{thm}
    There exist compact shrinking solitons to the gradient flow of \Sp-structures given by \eqref{GF}. More precisely, one shrinking soliton is given by the \Sp-structure on $\operatorname{SU}(3)$ given by $\Phi_t = \frac{1}{4}(2-3t)^2\Phi_0$, described in Example \ref{ExampleSU3Flow}.
Moreover, the one-parameter family of \Sp-structures 
\begin{equation}
\begin{split}
    \Phi_t = &\frac{k^2}{(2t+1)^2}(e^{1234}-e^{1256}-e^{1357}- e^{1467}-e^{3456}+e^{2457}-e^{2367})\\
    &+\frac{k^2}{(2t+1)}(- e^{1278}+e^{1368}-e^{1458}+e^{5678} -e^{3478}- e^{2468}-e^{2358}).
\end{split} 
\end{equation}
on the $T^7$-bundle over $S^1$ described in Example \ref{ExampleOtherSolution} is a solution to the gradient flow of \Sp-structures, along which the $T^7$-fibres shrink uniformly and the $S^1$-base expands.
\end{thm}

Having found an explicit example of a compact shrinking soliton to the gradient flow, it is natural to consider the question of whether this soliton is a stable critical point of the (suitably renormalised) flow. Schiemanowski's stability result \cite[Theorem 2]{schiemanowski2017stabilityspinorflow} does not apply here, since our manifold \emph{does} admit Killing fields.
Indeed, in Section \ref{SectionStability}, we obtain the following result.

\begin{thm*}[\ref{ExampleSU3Unstable}]
    The \Sp-soliton constructed in Example \ref{ExampleSU3Flow} is an unstable critical point for the renormalised gradient flow \eqref{EquationRescaledFlow}, modulo rescaling. Moreover, there exist stable directions, unstable directions and zero modes.
\end{thm*}

\bigskip
\noindent\textbf{Acknowledgements.}
The author would like to thank his supervisor Jason D. Lotay for his constant support and many helpful discussions that led to this article. This work was supported by a doctoral scholarship from the Engineering and Physical Sciences Research Council (Project reference EP/W524311/1 2929148).

\section{\texorpdfstring{Preliminaries on \Sp-structures}{Preliminaries on Spin(7)-structures}}\label{SectionPrelims}

In this section, we recall some necessary information about \Sp-structures, primarily to fix notation and to introduce the objects of interest in this paper. More detail can be found in \cite{JoyceBook} and, particularly pertaining to flows, \cite{KarigiannisFlows} and \cite{Dwivedi24}. We also make explicit the well-known relationship between torsion forms and torsion tensors, using the expression for the torsion tensor in terms of the intrinsic torsion from \cite{Niedzialomski}, which will be useful when dealing with explicit computations in Sections \ref{SectionReformulation} and \ref{SectionExamples}.

We begin by defining the Lie group Spin($7$) and \Sp-structures.
\begin{defn}
    Let $\mathbb{R}^8$ have an orthonormal basis given by $\{e_1,\cdots, e_8\}$ and define the $4$-form $\Phi_0$ as
    \begin{equation}
        \begin{split}
            \Phi_0 =& e^{1234}+ e^{1256} + e^{1278} + e^{1357} - e^{1368} - e^{1458} - e^{1467}\\ 
            &+e^{5678}+ e^{3478} + e^{3456} + e^{2468} - e^{2457} - e^{2367} -e^{2358},
        \end{split}
    \end{equation}
    where $e^i$ denotes the $1$-form dual to $e_i$, and $e^{ijkl} = e^i \wedge e^j \wedge e^k \wedge e^l$. We define the Lie group Spin($7$) to be the subgroup of $\operatorname{GL}(8,\mathbb{R})$ that fixes $\Phi_0$.
\end{defn}

A \textit{Spin($7$)-structure} on an $8$-manifold $M$ is a reduction of the structure group of the frame bundle $Fr(M)$ from GL($8,\mathbb{R})$ to Spin($7$) $\subset$ SO($8$) $\subset$ GL($8,\mathbb{R}$). For what follows, we will almost always use the following equivalent definition.

\begin{defn}\label{DefSPin7Structure}
    A \Sp-structure on an $8$-manifold $M$ is a choice of $4$-form $\Phi$ on $M$ such that, for each $p\in M$, there exists an oriented isomorphism between $T_pM$ and $\mathbb{R}^8$ for which $\Phi|_p$ is identified with the $4$-form $\Phi_0$ on $\mathbb{R}^8$ defined above. Such a $4$-form $\Phi$ is called \emph{admissible} and referred to as a \emph{Cayley form} or a \emph{\Sp-form}.
    The space of admissible $4$-forms is denoted $\mathcal{A}M.$
\end{defn}

Since Spin($7$) is a subgroup of $\operatorname{SO}(8)$, the existence of a \Sp-structure implies, in particular, the existence of a reduction of $Fr(M)$ to SO($8$) which provides orientability and existence of a Riemannian metric. The induced Riemannian metric is explicitly described in \cite{KarigiannisDeformations}, but we will not need such a formula here. We will denote the Riemannian metric induced by a \Sp-structure $\Phi$ by $g_\Phi$, and the associated Hodge star by $*_\Phi$. We have that $*_\Phi \Phi = \Phi$, and that $\Phi \wedge \Phi = 14\operatorname{vol}_\Phi$. The question of whether a given $8$-manifold admits a \Sp-structure is a topological one, and boils down to the following criteria.

\begin{prop}[\cite{LawsonMichelson}]\label{PropTopologicalCondition}
    An $8$-manifold $M$ admits a \Sp-structure if and only if both of the following criteria hold.
    \begin{enumerate}
        \item The first and second Stiefel-Whitney classes of $M$ vanish.
        \item Either one of the following: $p_1(M)^2 - 4p_2(M) \pm \chi(M)=0$,
    \end{enumerate}
    where $p_i(M)$ denotes the $i$th Pontryagin class of $M$, and $\chi(M)$ is the Euler characteristic.
\end{prop}
What is more interesting, and far more difficult, is finding so-called \emph{torsion-free} \Sp-structures.

\begin{defn}
    Let $(M,\Phi)$ be an $8$-manifold with \Sp-structure. We say $\Phi$ is torsion-free if
    \[
    \nabla^{g_\Phi}\Phi = 0,
    \]
    where $\nabla^{g_\Phi}$ is the Levi-Civita connection with respect to the Riemannian metric $g_\Phi$.
\end{defn}

We now note the following equivalent conditions, the first of which is the primary motivation behind finding torsion-free \Sp-structures.
\begin{thm}\label{equivalence}
    Let $(M,\Phi)$ be an $8$-manifold with \Sp-structure. The following are equivalent:
    \begin{enumerate}
        \item $Hol(g_\Phi) \subseteq \text{Spin}(7)$,
        \item $\nabla^{g_\Phi}\Phi = 0$,
        \item $\mathrm{d}\Phi = 0$,
    \end{enumerate}
    where $Hol(g_\Phi)$ is the holonomy group of the Riemannian metric $g_\Phi$ induced by $\Phi$.
\end{thm}

A manifold with holonomy contained in Spin($7$)  is called a \emph{Spin(7)-manifold}.
We mention one of the most important consequence of having holonomy contained in \Sp.

\begin{thm}[\cite{Bonan}]\label{TheoremRicciFlat}
    Let $(M,\Phi)$ be an $8$-manifold with \Sp-structure. If $Hol(g_\Phi)\subseteq \text{Spin}(7)$, then $(M,g_\Phi)$ is Ricci-flat.
\end{thm}

We note that \Sp-manifolds currently provide the only non-trivial, non-complex, compact examples of Ricci-flat $8$-manifolds, which is one reason why their study is interesting.

\subsection{Decomposition of the space of forms}\label{SubsectionFormDecompositions}
The existence of a \Sp-structure $\Phi$ on an $8$-manifold $M$ defines an action of Spin($7$) on the spaces of differential forms, and so induces a decomposition of the space of differential forms on $M$ into irreducible \Sp-representations.

\begin{prop}[{\cite[Prop. 11.4.4]{JoyceBook}}]\label{PropDecompositionForms}
    Let $(M,\Phi)$ be an $8$-manifold with \Sp-structure. Then, the spaces $\Lambda^kT^*M$ decompose orthogonally (with respect to the metric on forms induced by $g_\Phi$) as follows, where $\Lambda^k_l$ corresponds to an irreducible \Sp-representation of dimension $l$:
    
    \begin{equation}\label{EquationFormSplitting}
  \begin{split}
    &\Lambda^1T^*M = \Lambda^1_8,\\
    &\Lambda^2T^*M  = \Lambda^2_7 \oplus \Lambda^2_{21},\\
    &\Lambda^3T^*M  = \Lambda^3_8 \oplus \Lambda^3_{48},\\
    &\Lambda^4T^*M  = \Lambda^4_1 \oplus \Lambda^4_7 \oplus \Lambda^4_{27} \oplus \Lambda^4_{35}.
\end{split}  
\end{equation}
Moreover, the Hodge star defines an isomorphism \[ *_\Phi : \Lambda^k_l \to \Lambda^{n-k}_l,
\]
so the above decompositions also give decompositions for $\Lambda^5T^*M ,\cdots,\Lambda^8T^*M $.
In particular,

\begin{equation} \label{Equation5Forms}
    \Lambda^5T^*M  = \Lambda^5_8 \oplus \Lambda^5_{48}.
\end{equation}

\end{prop}
In what follows, we will write $\Omega^k(M)$ for the space of $k$-forms on $M$ (i.e., $\Omega^k(M) = \Gamma(\Lambda^kT^*M)$), or simply $\Omega^k$. In the presence of a \Sp-structure, each space of forms $\Omega^k$ decomposes as above, e.g.,
\begin{equation}
    \Omega^4 = \Omega^4_1 \oplus \Omega^4_7 \oplus \Omega^4_{27}\oplus \Omega^4_{35}.
\end{equation}
Although the metric $g_\Phi$ and the Hodge star $*_\Phi$ depend on $\Phi$, we will often write simply $g$ and $*$.
We can describe some of the irreducible subspaces more explicitly, in terms of $\Phi$, as follows \cite[Section 4.2]{KarigiannisDeformations}:

\begin{align}\label{forms}
  \Omega^2_7   &= \{\alpha \in \Omega^2 \mid *(\alpha \wedge \Phi) = 3\alpha\}, 
  &\Omega^2_{21} &= \{\alpha \in \Omega^2 \mid *(\alpha \wedge \Phi) = -\alpha\}, \\[4pt]
  \Omega^3_8   &= \{*(\alpha \wedge \Phi) \mid \alpha \in \Omega^1_8\}, 
  &\Omega^3_{48} &= \{\alpha\in \Omega^3 \mid \alpha \wedge \Phi = 0\}, \\[4pt]
  \Omega^4_1   &= \{\lambda \Phi \mid \lambda \in \mathbb{R}\}, 
  &\Omega^4_{35} &= \{\alpha \in \Omega^4 \mid *\alpha = -\alpha\}.
\end{align}

Again, applying the Hodge star gives, for example, 

\begin{equation}\label{EquationLambda58}
    \Omega^5_{8} = \{\alpha \wedge \Phi \mid \alpha \in \Omega^1_8 \}.
\end{equation}

The equations above allow us to write down projection formulae for $\pi^k_l:\Omega^k \to \Omega^k_l$, as well as expressions for the decompositions in local coordinates. The ones we will need are listed in the following proposition.

\begin{prop}[{\cite[Prop. 2.1]{KarigiannisFlows}}]
    Writing $\pi^k_l:\Omega^k \to \Omega^k_l$ for the projection maps from the space of $k$-forms to the (pointwise) $l$-dimensional irreducible component described above, we have the following explicit formulae:
    
\begin{align}
  \pi^2_7(\alpha)   &= \frac{\alpha + *(\Phi \wedge \alpha)}{4}, 
  &\pi^2_{21}(\alpha) &= \frac{3\alpha - *(\Phi \wedge \alpha)}{4}, \\[4pt]
  \pi^4_1(\alpha)   &= \frac{g_\Phi(\alpha,\Phi)}{|\Phi|^2}\,\Phi, 
  &\pi^4_{35}(\alpha) &= \frac{\alpha - *\alpha}{2}.  
\end{align}

\end{prop}
When there is little risk of ambiguity, we will simply write $\pi_l$ for the projection.
We will also write $\pi^k_{l\oplus m} : \Omega^k \to \Omega^k_{l}\oplus \Omega^k_m$ for the natural projections obtained from those defined above.

In order to describe the space of $4$-forms, and therefore deformations and flows of \Sp-structures, Karigiannis introduces the \emph{diamond} map, which we recall here. (Note that Karigiannis calls this map $D$ \cite{KarigiannisFlows}. Here, we adopt the notation $\diamond$, as used more recently in, for example, \cite{Dwivedi24} \cite{DLE24}).
Let $A$ be a $(0,2)$-tensor and consider the following map:
\begin{equation}
    A \mapsto(A \diamond \Phi)_{ijkl} = \left( A_{ip} g^{pq} \Phi_{qjkl} + A_{jp} g^{pq} \Phi_{iqkl} + A_{kp}g^{pq}  \Phi_{ijql} + A_{lp} g^{pq} \Phi_{ijkq} \right). 
\end{equation}

The diamond map $\diamond$ depends on the metric $g_\Phi$ and therefore on $\Phi$, but we will simply write $\diamond$ to avoid an overload of notation.
Using the metric, $A$ can be viewed as a $(1,1)$-tensor, and under the identification $\Gamma(TM\otimes T^*M) \cong \operatorname{End}(TM) \cong \mathfrak{gl}(8,\mathbb{R})$, the diamond map describes the infinitesimal action of $\operatorname{GL(8,\mathbb{R})}$ on $\Phi$. Hence, any infinitesimal deformation of a \Sp-structure $\Phi$ can be written as $A\diamond\Phi$ for some $A$. We can decompose $A$ into symmetric and skew-symmetric parts : $A = h + X$, for $h\in S^2$ and $X \in \Omega^2$. Moreover, Karigiannis shows that the kernel of $\diamond$ is $\Omega^2_{21}$, so we can take $X$ to lie in $\Omega^2_7$.
\begin{prop}[{\cite[Proposition 2.3]{KarigiannisFlows}}]\label{PropDiamondKernel}
    The kernel of the map $A \mapsto A \diamond \Phi$ is isomorphic to $\Omega^2_{21}$.
\end{prop}
One can also show the following.
\begin{cor}[{\cite[Corollary 2.6]{KarigiannisFlows}}]\label{CorollaryIsomorphism}
    The map $A \mapsto A\diamond \Phi$ is injective on $S^2 \oplus \Omega^2_7$, and is therefore an isomorphism onto its image $\Omega^4_1 \oplus \Omega^4_7 \oplus \Omega^4_{35} $. 
    Decompose the space of symmetric $2$-tensors $S^2$ into multiples of the metric and trace-free parts: $S^2 = \langle g_\Phi \rangle  \oplus S^2_0$.
    The summands $\langle g_\Phi \rangle,S^2_0$ and $\Omega^2_7$ are mapped isomorphically onto $\Omega^4_1,\Omega^4_{35}$ and $\Omega^4_7$, respectively.
\end{cor}
In Section \ref{SectionFlows} we will see how to use the above discussion to describe all flows of \Sp-structures. Before that, we use it to define the \emph{torsion tensor}, which we then relate to the \emph{torsion forms}.

\subsection{The torsion tensor and torsion forms of a \texorpdfstring{\Sp-structure}{Spin(7)-structure}}
The obstruction to a \Sp-structure being torsion-free is called \emph{torsion}. In light of Theorem \ref{equivalence}, there are  two equivalent ways of defining and working with the torsion of a \Sp-structure, using $\mathrm{d}\Phi$ or $\nabla\Phi$. Considering $\mathrm{d}\Phi$ and how it decomposes into components will give rise to the notion of \emph{torsion forms}, and studying $\nabla \Phi$ will give rise to the notion of the \emph{torsion tensor}. In this section, we recall both notions, and make concrete the well-known equivalence between them by explicitly expressing the torsion tensor in terms of the torsion forms.
We start with the torsion tensor, where our discussion follows \cite[Section 2.2]{KarigiannisFlows}.

\begin{lem}[{\cite[Lemma 2.10]{KarigiannisFlows}}]
Let $X$ be a vector field and $\Phi$ be a \Sp-structure on an $8$-manifold $M$.
Then, $\nabla_X\Phi \in \Omega^4_7 \subset \Omega^4$. Thus, the $(0,5)$-tensor $\nabla \Phi$ lies in the space $\Omega^1_8 \otimes \Omega^4_7$.   
\end{lem}

Because of this, Corollary \ref{CorollaryIsomorphism} implies that for each vector field $e_m$, there exists a $2$-form $T_m\in \Omega^2_7$ (where $m$ is fixed) such that $\nabla_m\Phi = T_m\diamond\Phi$. This motivates the following definition.

\begin{defn}[{\cite[Definition 2.12]{KarigiannisFlows}}] \label{DefTorsion}
    The \emph{torsion tensor} $T$ of a \Sp-structure $\Phi$ is the element of $\Omega^1_8 \otimes \Omega^2_7$ such that
    \begin{equation}\label{EquationDefTorsion}
        \nabla_m\Phi_{ijkl} = (T_m \diamond \Phi)_{ijkl} = T_{m;ip} g^{pq} \Phi_{qjkl} 
+ T_{m;jp} g^{pq} \Phi_{i q k l} 
+ T_{m;kp} g^{pq} \Phi_{ij q l} 
+ T_{m;lp} g^{pq} \Phi_{ijk q}.
    \end{equation}
\end{defn}
Note that the semi-colon in $T_{m;ab}$ does not denote covariant differentiation of $T_m$. Rather, for each fixed $m$, $T_{m;ab}$ lies in $\Omega^2_7$ and the semi-colon serves only to separate the $\Omega^1_8$-index $m$ from the two $\Omega^2_7$-indices $a,b$.

We can explicitly describe the tensor $T_{m;ab}$ in local coordinates, in terms of $\Phi$, as follows. 
\begin{prop}[{\cite[Lemma 2.13]{KarigiannisFlows}}]\label{PropTorsion}
    The torsion tensor $T$ of a \Sp-structure $\Phi$ can be expressed in local coordinates as
    \begin{equation}\label{EquationTorsionTensor}
        T_{m;ab} = \frac{1}{96}(\nabla_m \Phi_{ajkl})(\Phi_{bpqr})g^{jp}g^{kq}g^{lr}.
    \end{equation}
\end{prop}

For expressing other geometric quantities in terms of torsion, as well as describing the gradient flow of \Sp-structures, it will be useful to decompose the torsion tensor $T$ into two orthogonal components. We start with the following result, which is proved in \cite[p.~8]{Dwivedi24}.

\begin{prop}\label{PropTorsionSplitting}
    Let $M$ be an $8$-manifold with \Sp-structure $\Phi$.
    Then,
    \begin{equation}\label{EquationT8T48}
        \Omega^1_8 \otimes \Omega^2_7 \cong \Omega^1_8 \oplus_\perp\{\gamma_{i;jk} \in \Omega^1_8 \otimes \Omega^2_7 \mid \gamma_{i;jk}g^{ik} = 0\} \cong \Omega^1_8 \oplus \Omega ^3_{48}.
    \end{equation}
    
\end{prop}
Thus, the torsion tensor $T$ decomposes as $T = T_8 + T_{48}$. The dimensions of the components of this splitting coincide with those of the splitting $\Omega^5 = \Omega^5_8 \oplus\Omega^5_{48}$ (Proposition \ref{PropDecompositionForms}), alluding to the aforementioned relationship between $T$ and $\mathrm{d}\Phi$. Note also that \eqref{EquationT8T48} gives that 
\begin{equation}
    (T_8)_j = T_{i;jk}g^{ik}.
\end{equation}
Here, $T_8$ is defined as a $1$-form, but we will also write $T_8$ for the associated dual vector field. We will also need the following definition of the \emph{divergence} of the torsion tensor.
\begin{defn}
    Let $T$ be the torsion tensor of a \Sp-structure $\Phi$ on a manifold $M$. We write $\operatorname{div}T$ for the divergence of the torsion, which is an element of $\Omega^2_7$ and is defined by 
    \begin{equation}
        \operatorname{div}T_{jk} = g^{nm}\nabla_nT_{m;jk}.
    \end{equation} 
\end{defn}
Here, we also mention the following explicit formula for the Ricci tensor of the metric induced by a \Sp-structure $\Phi$ in terms of its torsion tensor $T$, which also illustrates the fact that torsion-free \Sp-structures induce Ricci-flat metrics.
\begin{prop}[{\cite[Proposition 4.6]{KarigiannisFlows}}]\label{PropRicciTensor}
    Let $T$ be the torsion tensor of a \Sp-structure $\Phi$ on a manifold $M$. Then, the Ricci tensor of the metric $g_\Phi$ induced by $\Phi$ can be expressed as
    \begin{equation}\label{EquationRicci}
        R_{ij} = 4\nabla_i(g^{ap}T_{a;jp}) - 4g^{ap}\nabla_aT_{i;jp} - 8T_{i;jb}g^{ap}g^{bq}T_{a;qp} + 8T_{a;jb}g^{ap}g^{bq}T_{i;qp}.
    \end{equation}
\end{prop}
We now turn to \emph{torsion forms}. A \Sp-structure $\Phi$ is \emph{torsion-free} if and only if $\mathrm{d}\Phi = 0$, so the torsion of a \Sp-structure is captured in the $5$-form $\mathrm{d}\Phi$. By \eqref{Equation5Forms}, $\mathrm{d}\Phi = (\mathrm{d}\Phi)_8 +(\mathrm{d}\Phi)_{48}$, and \eqref{EquationLambda58} tells us that we can write $(\mathrm{d}\Phi)_8 = \alpha \wedge \Phi$ for some $\alpha \in \Omega^1_8$. We denote this $\alpha$ by $T^1_8$ and define $T^5_{48} \coloneq (\mathrm{d}\Phi)_{48}$.
\begin{defn}
Let $\Phi$ be a \Sp-structure.
    Expressing \begin{equation}
        \mathrm{d}\Phi = T^1_8 \wedge \Phi + T^5_{48},
    \end{equation}
    as above, we call the forms $T^1_8$ and $T^5_{48}$ the \emph{torsion forms} of the \Sp-structure $\Phi$.
\end{defn}

Note that, by Theorem \ref{equivalence}, a \Sp-structure is torsion-free if and only if both torsion forms vanish, but the refinement of having two torsion forms rather than one torsion tensor also allows us to consider the cases when just one of the torsion forms vanishes. 
\begin{defn}
    Let $\Phi$ be a \Sp-structure with $\mathrm{d}\Phi = T^1_8 \wedge \Phi + T^5_{48}$. If $T^1_8 = 0$, we say that $\Phi$ is \emph{balanced}, and if $T^5_{48} = 0$, we say that $\Phi$ is \emph{locally conformally parallel}. 
\end{defn}
We now have two notions of torsion, arising from considering $\mathrm{d}\Phi$ and $\nabla \Phi$. By Theorem \ref{equivalence}, these two notions contain essentially the same information, and we make this concrete by expressing the torsion tensor in terms of the two torsion forms.

We first show that the torsion form $T^1_8$ is a constant multiple of the $1$-form $(T_8)_j \coloneq T_{i;jk}g^{ik}.$ The fact that they are constant multiples of each other is well-known, but does not seem to have appeared in the literature, and here we explicitly compute the constant factor. This will be useful when we deal with explicit examples of flows in Section \ref{SectionExamples}, as $T^1_8$ is easier to compute than $T_8$, and we will actually have already computed $T^1_8$ along the way to computing the torsion tensor $T$.
\begin{prop}\label{PropT18T8}
 Let $M$ be a manifold with \Sp-structure $\Phi$, torsion tensor $T$ and torsion $1$-form $T^1_8$. The $T_8$-component of $T$ (see Proposition \ref{PropTorsionSplitting}), can be expressed as 
 \begin{equation}
     T_8 = -\frac{7}{16}T^1_8.
 \end{equation}
 \end{prop}
\begin{proof}
Let $\boldsymbol{k}$ denote the k-dimensional irreducible \Sp-representation.
    Recall that the torsion tensor $T$ lies in the space $\Omega^1_8 \otimes \Omega^2_7$, and that we can view this space as a \Sp-representation isomorphic to $\boldsymbol{8}\oplus \boldsymbol{48}$ \cite{Dwivedi24}. Moreover, the $5$-form $\mathrm{d}\Phi = T^1_8 \wedge \Phi + T^5_{48}$ lies in $\Omega^5 \cong \boldsymbol{8} \oplus \boldsymbol{48}$. Now, by the definition of $T_8$ as the trace given in Proposition \ref{PropTorsionSplitting}, we can view $T_8$ as a linear \Sp-equivariant map $\boldsymbol{8}\oplus \boldsymbol{48} \to \boldsymbol{8}$. Likewise, we can view $T^1_8$ as such a map and, by Schur's lemma, the space of such maps is one-dimensional. So, \begin{equation}
        T_8 = cT^1_8,
    \end{equation}
    for some universal constant $c$, which can be determined by computing both terms for a single \Sp-structure. More precisely, we compute both terms for the \Sp-structure given in \cite[Example 4.8]{Niedzialomski}, with respect to an orthonormal basis of $1$-forms $e^1,\cdots,e^8$. We compute that $T^1_8 = \frac{2}{7}e^6$ and $T_8 = -\frac{1}{8}e^6$, which gives $T_8 = -\frac{7}{16}T^1_8$, as required.
\end{proof}

We now describe how to express the full torsion tensor $T$ in terms of the two torsion forms $T^1_8$ and $T^5_{48}$, by expressing the results of \cite[Proposition 2.2 and Section 2.3.2]{Niedzialomski} in terms of torsion forms.
\begin{prop}
    Let $M$ be an $8$-manifold with \Sp-structure $\Phi$ and a local orthonormal frame $\{e_1,\cdots,e_8\}$.
    Then, 
    \begin{equation}
        T_{m;ab} = \left(\frac{1}{2}(e_m \lrcorner T^c)_7\right)_{ab},
    \end{equation}
    where $T^c = -\frac{1}{6}*(T^1_8 \wedge \Phi) + *T^5_{48}$ and $\pi_7$ denotes the projection from $\Omega^2$ to $\Omega^2_7$.
\end{prop}

\begin{proof}
    In \cite[Section 2.3.2]{Niedzialomski}, Niedzialomski shows the above result, but with $T^c$ defined as 
    \begin{equation}
        T^c = -\delta \Phi - \frac{7}{6} * (\theta \wedge \Phi), \qquad
\theta = \frac{1}{7} * (\delta \Phi \wedge \Phi),
    \end{equation}
    where $\theta$ is called the \emph{Lee form} of $\Phi$. Here, we show that $\theta = T^1_8$, which gives the result.
     We have the following identity for any $1$-form $\eta \in \Omega^1$(\cite[Proposition 4.2.1]{KarigiannisDeformations}):
    \begin{equation}
        *(\Phi \wedge*(\Phi\wedge \eta)) = -7\eta.
    \end{equation}
    Now, $\mathrm{d}\Phi = T^1_8 \wedge \Phi +*T^5_{48}$ so 
    \[\delta\Phi = -*\mathrm{d}\Phi = -*(T^1_8 \wedge \Phi) + *T^5_{48}.\]
     Since $T^5_{48} \in \Omega^5_{48}$, \eqref{forms} gives that $  \Phi \wedge  *T^5_{48} = 0$, so 
    \begin{equation}
    \begin{split}
         *(\delta \Phi \wedge \Phi) &= -*(\Phi \wedge *d\Phi)\\
                                    &=-*(\Phi \wedge (*T^5_{48} +*(\Phi \wedge T^1_8)))\\
                                    &= -*(\Phi \wedge *(\Phi \wedge T^1_8))\\
                                    & = 7T^1_8.
    \end{split}
    \end{equation}
    Rearranging gives that $T^1_8 = \frac{1}{7}*(\delta\Phi \wedge \Phi) = \theta$, and substituting this into the equation for $T^c$ gives the result.
\end{proof}

Thi expression of the torsion tensor in terms of torsion forms is more amenable to explicit computation than the formula given by \eqref{EquationTorsionTensor}, and will be useful in Section \ref{SectionExamples} when we compute explicit examples of the torsion tensor in order to find examples of solutions to various flow equations. Before coming to that, we discuss general flows of \Sp-structures, and introduce the particular flow that we are interested in.

\subsection{A gradient flow of \texorpdfstring{\Sp-structures}{Spin(7)-structures}}\label{SectionFlows}

Geometric flows of \Sp-structures were first studied in \cite{KarigiannisFlows}, where Karigiannis considers the most general flow of \Sp-structures, using the discussion of arbitrary deformations of \Sp-structures that we outlined in Section \ref{SubsectionFormDecompositions}.
Recall that an arbitrary infinitesimal deformation of a \Sp-structure $\Phi$ can be written as $A \diamond \Phi$ for some $A\in S^2 \oplus \Omega^2_7$. Thus, an arbitrary flow of \Sp-structures can be written as 
\begin{equation}\label{EquationGeneralFlow}
    \frac{\partial}{\partial t}\Phi(t) = (A(t) \diamond_t \Phi(t)) = ((h(t)+X(t))\diamond_t \Phi(t) ),
\end{equation}
where $A(t)$ is a one-parameter family of tensors in $S^2 \oplus \Omega^2_7$ and the symbol $\diamond_t$ refers to the fact that the contraction defined by $\diamond_t$ is with respect to the metric $g_{\Phi_t}$ induced by the \Sp-structure $\Phi_t$. We will often write $A_t$ instead of $A(t)$ and $\Phi_t$ instead of $\Phi(t)$. To avoid an overload of notation, we will often omit writing the $t$-dependence of tensors altogether, and note that unless otherwise stated, all terms appearing in the flow equations that follow are dependent on $t$. Karigiannis also computes the induced evolution equations of various metric and \Sp-related quantities under the most general flow of \Sp-structures \cite[Section 3]{KarigiannisFlows}.
In particular, we have that 
\begin{equation}
  \frac{\partial}{\partial t}g_{ij} = 2h_{ij}  
\end{equation}
and
\begin{equation}
  \frac{\partial}{\partial t} T_{m;\alpha\beta}
= A_{\alpha p} \, g^{pq} T_{m;q\beta}
- A_{\beta p} \, g^{pq} T_{m;q\alpha}
+ \pi_7 \left( \nabla_\beta h_{\alpha m}
- \nabla_\alpha h_{\beta m}
+ \nabla_m X_{\alpha\beta} \right),  
\end{equation}
from which the evolution equations of all other metric and torsion-related quantities can be obtained. In particular, this shows that only the symmetric part of $A$ affects the induced metric, and this combined with Corollary \ref{CorollaryIsomorphism} shows that deformations in the $\Omega^4_7$-direction do not affect the metric.

Motivated by finding torsion-free \Sp-structures, it is natural to consider specific flows which have torsion-free \Sp-structures as critical points. Perhaps the most natural flow to consider is a gradient flow of the norm of the torsion tensor, in an attempt to decrease the torsion as quickly as possible. This is the motivation behind the so-called gradient flow of \Sp-structures, introduced by Dwivedi \cite{Dwivedi24}, which we recall now. All details and derivations can be found in Sections $3$ to $5$ of \cite{Dwivedi24}.

\begin{defn}
    Given a compact $8$-manifold with Spin($7$)-structure, $(M, \Phi)$, the energy functional $E$ is defined as 
    \begin{equation}
        E(\Phi) = \frac{1}{2}\int_M \lvert T_\Phi \rvert^2 \mathrm{vol}_\Phi,
    \end{equation}
    where $T_\Phi$ is the torsion tensor of the Spin($7$)-structure $\Phi$, and the norm and volume form are those induced by the Riemannian metric $g_\Phi$. To save on notation, we will write $T$ for $T_\Phi$.
\end{defn}
Dwivedi computes the gradient flow of this energy functional, obtaining the following evolution equation for $\Phi(t)$.

\begin{defn}[\cite{Dwivedi24}]
    Let $(M,\Phi_0)$ be a compact $8$-manifold with initial Spin($7$)-structure $\Phi_0$. The gradient flow of \Sp-structures is the following initial value problem:

    \begin{equation}
            \begin{cases}
                 \frac{\partial}{\partial t} \Phi(t) = (-\operatorname{Ric} + 2(\mathcal{L}_{T_8}g) + T\star T - |T|^2g + 2 \text{ div}T)_t \diamond_t \Phi(t),\\
                 \Phi(0) = \Phi_0, \\
                 
            \end{cases} \tag{GF}
        \end{equation}
        where each term is induced by the Spin($7$) structure $\Phi_t$
        and 
        \begin{equation}
            (T \star T)_{ij} = 4T_{b;il}T_{j;lb} + 4T_{b;jl}T_{i;lb} - 4T_{j;il}T_{b;lb} - 4T_{i;jl}T_{b;lb} + 2T_{i;lb}T_{j;lb}.
        \end{equation}
        This flow is the negative gradient flow of the functional $2E(\Phi)$.
\end{defn}
Dwivedi proves short-time existence of this flow in \cite{Dwivedi24}, as well as non-existence of compact expanding solitons. 
So far, no explicit solutions have been found. One difficulty with finding explicit solutions to this flow is the computational difficulties of evaluating tensorial quantities like $\operatorname{div}T$ and $\mathcal{L}_{T_8}g$. To avoid some of this computational difficulty, we express the terms of the above flow equation in terms of \emph{torsion forms} (Section \ref{SectionReformulation}), which are easier to compute, and then use this formulation to construct explicit solutions (Section \ref{SectionExamples}). In doing so, we will use the soliton equation and rescaling properties obtained in \cite{Dwivedi24}, which we mention now.
\begin{defn}
    A soliton solution for \eqref{GF} on an $8$-manifold $M$ is a triple $(\Phi(t),Y(t),\lambda(t))$, where $Y \in \Gamma(TM)$ and $\lambda \in \mathbb{R}$ such that 
    \begin{equation}\label{EquationSoliton}
     (-\operatorname{Ric} + 2(\mathcal{L}_{T_8}g) + T\star T - |T|^2g + 2\operatorname{div}T)_t\diamond_t \Phi(t) = \lambda \Phi(t) + \mathcal{L}_Y \Phi(t).   
    \end{equation}
\end{defn}
Solitons are of particular interest since they are in one-to-one correspondence with self-similar solutions, i.e.,
\begin{equation}
    \Phi(t) = \lambda(t)^4 f(t)^* \Phi(0),
\end{equation}
where $\lambda(t)$ are scalings and $f(t) : M \rightarrow M$ are diffeomorphisms.

For many geometric flows, solitons are often local models for singularities, and so studying solitons provides information about the possible geometry around the formation of singularities. Studying this for the gradient flow of \Sp-structures would be interesting future work.

A soliton is called \emph{expanding} if $\lambda >0$, \emph{steady} if $\lambda =0$ and \emph{shrinking} if $\lambda <0$.
In discussing the effect of rescaling the \Sp-structure, to avoid fractional powers of the scale factor, we rescale to $\tilde\Phi = c^4\Phi$ for some $c>0$ constant, and decorate any term induced by $\tilde{\Phi}$ with its own tilde. Note that $\tilde \Phi$ also defines a \Sp-structure, and that this rescales the metric by $\tilde g = c^2g$ and so $\tilde{g}^{-1} = c^{-2}g^{-1}$. Moreover, from the definition of the torsion tensor (Definition \ref{DefTorsion}), and using that the diamond operator involves the inverse metric, we see that $\tilde{T} = c^2T$. Using this and writing  $A = (-\operatorname{Ric} + 2(\mathcal{L}_{T_8}g) + T\star T - |T|^2g + 2 \text{ div}T)$, we see that $\tilde{A} = A$ and so 
\begin{equation}\label{EquationRescaledFlow}
    \tilde{A} \tilde{\diamond} \tilde{\Phi}  = c^2(A \diamond \Phi).
\end{equation}
We will refer back to this when discussing explicit examples of solitons in Section \ref{SectionExamples}.

\subsection{Other flows of \texorpdfstring{\Sp-structure}{Spin(7)-structures}}

There are two other flows of \Sp-structures that have received attention in the literature, and we briefly mention each here. The first is the so-called \emph{harmonic flow}, introduced in \cite{DLE24}, and given by the following initial value problem:
\begin{equation}
  \begin{cases}
\displaystyle \frac{\partial \Phi}{\partial t} = \operatorname{div} T \diamond \Phi, \\
\Phi(0) = \Phi_0.
\end{cases}  
\end{equation}

This flow enjoys several nice analytic properties, but does not evolve the induced metric $g_\Phi$, since $\operatorname{div}T \in \Omega^2_7$. For this reason, the harmonic flow is unlikely to be useful in finding torsion-free \Sp-structures, since in doing so it would be necessary to start with an initial \Sp-structure whose induced metric was already Ricci-flat. Nevertheless, it may still lead to interesting information about the space of \Sp-structures, and since the harmonic flow has nice properties, it is natural to consider a modification for which the induced metric \emph{does} evolve. A candidate is the \emph{Ricci-harmonic flow},  introduced in \cite[Remark 4.13]{Dwivedi24}, which is given by the following initial value problem:
\begin{equation}\label{EquationRicciHarmonic}
 \begin{cases}
\displaystyle \frac{\partial \Phi}{\partial t} = (-\operatorname{Ric}+\operatorname{div} T )\diamond \Phi, \\
\Phi(0) = \Phi_0.
\end{cases}   
\end{equation}

The Ricci-harmonic flow induces precisely the Ricci flow on the induced metric, and it also has short-time existence and uniqueness \cite{Dwivedi24}. In section \ref{SectionExamples}, we will also obtain the first explicit solutions to this flow.

\section{A reformulation of the gradient flow}\label{SectionReformulation}

In Section \ref{SectionPrelims}, we showed how to express the torsion tensor $T$ of a \Sp-structure $\Phi$ in terms of its torsion forms $T^1_8$ and $T^5_{48}$, which are easier to compute. This gives us a tangible way to compute the torsion tensor in explicit examples, but in order to work with explicit examples of the gradient flow \eqref{GF}, we also need to be able to compute terms like $T\star T$ and $\mathcal{L}_{T_8}g$, which appear in the evolution equation \eqref{GF}. In this section, we use a representation theory approach to directly express these terms in terms of the torsion forms, without first computing the full torsion tensor. This method is originally due to Bryant \cite[Remark 10]{BryantRemarks} in the $G_2$-setting, and was used by Fowdar to express the Ricci tensor and scalar curvature in terms of \Sp-torsion forms in \cite[Section 4]{FowdarSpin7}.
 Here, we discuss Fowdar's argument and use it to express the other terms of the evolution equation \eqref{GF} in terms of torsion forms. We first fix some notation.
Let $S^2_0(\mathbb{R}^8)$ denote the space of traceless, symmetric $(0,2)$-tensors and define
\begin{equation}
    \begin{split}
        \mathsf{i}: \langle g_\Phi \rangle \oplus S^2_0(\mathbb{R}^8) &\to \Lambda^4_1 \oplus \Lambda^4_{35},\\
        a\circ b &\mapsto a \wedge *(b\wedge \Phi) + b \wedge *(a \wedge \Phi).
    \end{split}
\end{equation}
Note that this map is defined for pure symmetric tensors of the form $a \circ b$, where $a$ and $b$ are $1$-forms, and we extend it linearly to a map on the whole space of symmetric tensors.
We denote by $\mathsf{j}$ the inverse map of $\mathsf{i}$, extended to the zero map on $\Lambda^4_7 \oplus \Lambda^4_{27}$. 
Note that
\[
\mathsf{i} \circ \mathsf{j}(\alpha) = \pi^4_{1\oplus 35}(\alpha) = \frac{g_\Phi(\alpha,\Phi)}{| \Phi|^2}\Phi + \frac{\alpha -*\alpha}{2}.
\]

In \cite{FowdarSpin7}, Fowdar uses a representation-theoretic argument to determine a basis for the space of symmetric first- and second-order \Sp-invariants that can contribute to flows (i.e. \Sp-invariant tensors which depend on $\Phi$ and first and second derivatives of $\Phi$ and which arise as endomorphisms of the tangent bundle). We present this basis in the following proposition, which is essentially contained in the proof of \cite[Proposition 4.1]{FowdarSpin7} and \cite[Section 6]{Fowdar_SU2}, and we extend it to provide a basis for the skew-symmetric invariants. We also outline the method of proof, combining \cite{FowdarSpin7} and \cite{Fowdar_SU2}.

\begin{prop}
    Let $(M,\Phi)$ be a manifold with \Sp-structure. The space of first order \Sp-invariants is generated by:
    \begin{equation}
        T^1_8 \text{ and } T^5_{48}.
    \end{equation}
    The space of linear second order \Sp-invariants which contribute to flows is generated by:
    \begin{itemize}
        \item $\delta T^1_8g_\Phi,$

        \item $ \mathsf{j}(\delta(T^1_8 \wedge \Phi)),$
        \item $\mathsf{j}(\delta T^5_{48}),$
        \item $\pi_7(\mathrm{d}T^1_8).$
    \end{itemize}
    The space of quadratic first order \Sp-invariants which contribute to flows is generated by:
     \begin{itemize}
        \item $| T^1_{8}|^2g_\Phi,$
        \item $ | T^5_{48}|^2g_\Phi,$
        \item $\mathsf{j}((T^1_8 \wedge *T^5_{48})),$
        \item $\mathsf{j}(*(T^1_8 \wedge\Phi)\wedge T^1_8),$
        \item $g_\Phi(\cdot \lrcorner*T^5_{48},\cdot\lrcorner*T^5_{48}),$
        \item $\mathsf{j}((e_i \lrcorner *T^5_{48})\wedge (e_i \lrcorner *T^5_{48})),$ where $\{e_i\}_{i=1}^8$ is a local orthonormal basis,
        \item $\pi_7(*(T^1_8 \wedge T^5_{48})).$
    \end{itemize}
    
\end{prop}
\begin{proof}
    The space of second order \Sp-invariants $V_2(\mathfrak{spin}(7))$ was computed in \cite[Proposition 4.1]{Fowdar}. We are interested in those \Sp-invariants which can contribute to flows of \Sp-structures, which are those that arise as $\mathfrak{spin}(7)$-invariant endomorphisms of $\mathbb{R}^8$, since they correspond to the natural action of $\mathfrak{gl}(8,\mathbb{R}^8)$ on the defining $4$-form $\Phi$.
    To this end, the authors of \cite{Fowdar_SU2} write down that
    \begin{equation}
        V_2(\mathfrak{spin}(7)) \cap\operatorname{End}(\mathbb{R}^8)_{\mathfrak{spin}(7)} \cong \mathbb{R} \oplus \Lambda^2_7 \oplus 2S^2_0(\mathbb{R}^8),
    \end{equation}
    where $\operatorname{End}(\mathbb{R}^8)_{\mathfrak{spin}(7)}$ denotes the space of $\mathfrak{spin}(7)$-invariant endomorphisms of $\mathbb{R}^8$ and $S^2_0(\mathbb{R}^8)$ is the space of traceless symmetric $2$-tensors on $\mathbb{R}^8$. In \cite{FowdarSpin7}, Fowdar writes down explicit generators for the $\mathbb{R}$ and $S^2_0(\mathbb{R}^8)$ on the right hand side. This copy of $\mathbb{R}$ is generated by $\delta T^1_8g_\Phi$, and the two copies of $S^2_0(\mathbb{R}^8)$ are generated by $\mathsf{j}(\delta T^5_{48})$ and $\mathsf{j}(\delta(T^1_8 \wedge \Phi))$. Since the multiplicity of $\Lambda^2_7$ in the above decomposition is $1$, we only need one generator for this component. Any element of $\Lambda^2_7$ obtained from exterior derivatives of torsion forms will suffice, and we choose $\pi_7(\mathrm{d}T^1_8)$.
    The other terms which contribute to second order flows of \Sp-structures are the terms quadratic in the first-order \Sp-invariants. In \cite{FowdarSpin7}, Fowdar decomposes the space of first order \Sp-invariants $V_1$ as 
    \begin{equation}
        V_1(\mathfrak{spin}(7)) = V_{0,0,1} \oplus V_{1,0,1} \cong \langle T^1_8 \rangle \oplus \langle T^5_{48}\rangle , 
    \end{equation}
    and uses this to decompose the space $S^2(V_1(\mathfrak{spin}(7)))$ of invariants quadratic in first order invariants. Again, we are only interested in those that contribute to flows, and we obtain
    \begin{equation}
        S^2(V_1(\mathfrak{spin}(7))) \cap\operatorname{End}(\mathbb{R}^8)_{\mathfrak{spin}(7)} \cong 2\mathbb{R} \oplus \Lambda^2_7 \oplus 4S^2_0(\mathbb{R}^8).
    \end{equation}
    Again, Fowdar explicitly writes down generators for the copies of $\mathbb{R}$ and $S^2_0(\mathbb{R}^8)$. The two copies of $\mathbb{R}$ are generated by $|T^1_8|^2g_\Phi$ and $|T^5_{48}|^2g_\Phi$, and the four copies of $S^2_0(\mathbb{R}^8)$ are generated by $\mathsf{j}((T^1_8 \wedge *T^5_{48}))$, $\mathsf{j}(*(T^1_8 \wedge\Phi)\wedge T^1_8)$, $g_\Phi(\cdot \lrcorner*T^5_{48},\cdot\lrcorner*T^5_{48})$ and $\mathsf{j}((e_i \lrcorner *T^5_{48})\wedge (e_i \lrcorner *T^5_{48}))$. Again, there is only one copy of $\Lambda^2_7$, and we choose $\pi_7(*(T^1_8 \wedge T^5_{48}))$ as a generator.
\end{proof}

Fowdar uses these results to express the Ricci and scalar curvatures in terms of torsion forms. From the above, one knows that $\operatorname{Ric}$ must be a linear combination of the symmetric terms above, and that $\mathrm{Scal}$ must be a linear combination of the scalar terms. Computing on a few examples to determine coefficients results in the following.

\begin{prop}[{\cite[Proposition 4.1]{FowdarSpin7}}]
    The Ricci and scalar curvatures of a manifold with Spin($7$)-structure $(M,\Phi)$ are given by
    \begin{equation}
    \begin{split}
        \operatorname{Ric}(g_\Phi) & = \left(\frac{5}{8}\delta T^1_{8} + \frac{3}{8}| T^1_{8}|^2 - \frac{2}{7}| T^5_{48}|^2\right)g_\Phi \\
        &+\mathsf{j} \left( -3 \delta(T^1_8 \wedge \Phi) +4\delta T^5_{48} - 2(T^1_8 \wedge *_\Phi T^5_{48}) -\frac{9}{4}*_\Phi (T^1_8 \wedge \Phi) \wedge T^1_8\right) \\
        &+ \frac{1}{2}g_\Phi(\cdot \lrcorner *_\Phi T^5_{48},\cdot \lrcorner *_\Phi T^5_{48} ),\\
        \operatorname{Scal}(g_\Phi) &= \frac{7}{2}\delta T^1_8 + \frac{21}{8}| T^1_{8}|^2 -\frac{1}{2}| T^5_{48}|^2.    
    \end{split}         
    \end{equation}
\end{prop}
Here, we use a very similar approach to express the other terms of the flow equation in terms of the torsion forms. 
\begin{prop}\label{PropReformulation}
    Each additional term of the flow equation \ref{GF} can be expressed in terms of torsion forms as follows.
    \begin{equation}
    \begin{split}
        | T|^2g_\Phi &= \left (\frac{7}{32}| T^1_{8}|^2 +\frac{1}{4}| T^5_{48}|^2\right )g_\Phi,\\
        T \star T &= \frac{7}{16}| T^1_{8}|^2g_\Phi + \mathsf{j} \left( -7 (T^1_8 \wedge *_{\Phi}T^5_{48}) +\frac{7}{8}*_{\Phi} (T^1_8 \wedge \Phi) \wedge T^1_8\right),
        \\
        \mathcal{L}_{T_8}g &= \frac{7}{32} (\delta T^1_8)g_\Phi + \mathsf{j}\left (-\frac{7}{4} \delta(T^1_8 \wedge \Phi) + \frac{7}{4}(T^1_8 \wedge *_{\Phi}T^5_{48}) -\frac{7}{4}*_{\Phi}(T^1_8 \wedge \Phi) \wedge T^1_8\right ), \\
        \operatorname{div}T &= \frac{7}{12}\pi_7(\mathrm{d}T^1_8) + \frac{7}{24}\pi_7(*_{\Phi}(T^1_8 \wedge T^5_{48})).
    \end{split}
    \end{equation}
\end{prop}

\begin{proof}
    We proceed exactly as in \cite[Proposition 4.1]{FowdarSpin7}, but in this case there are a few shortcuts. The norm of the torsion tensor $|T|^2$ does not involve any derivatives of torsion, so must be expressible in terms of $|T^1_8|^2$ and $|T^5_{48}|^2$. Likewise, $T\star T$ does not involve any derivatives, so we need only consider the terms algebraic in torsion above. Finally, $\mathcal{L}_{T_8}g$ must vanish if $T_8 = 0$ (equivalently $T^1_8 =0$ by Proposition \ref{PropT18T8}), so we need not consider the terms only involving $T^5_{48}$. To determine coefficients, we compute on a few examples.
\end{proof}

Combining the above, we obtain the following expression for the gradient flow of \Sp-structures in terms of torsion forms.
\begin{thm}\label{ThmTorsionForms}
    Writing the gradient flow of \Sp-structures \eqref{GF} as 
    \[
    \frac{\partial}{\partial t}\Phi_t = A \diamond \Phi,
    \]
    we can express $A$ as
    \begin{equation}
    \begin{split}
        A& = \left( -\frac{3}{16}\delta T^1_{8} - \frac{5}{32}| T^1_{8}|^2 +  \frac{1}{28}| T^5_{48}|^2\right)g_\Phi \\
        &+\mathsf{j} \left(  -\frac{1}{2}\delta(T^1_8 \wedge \Phi) -4\delta T^5_{48}   -\frac{3}{2}(T^1_8 \wedge *_\Phi T^5_{48}) -\frac{3}{8}*_\Phi (T^1_8 \wedge \Phi) \wedge T^1_8\right) \\
        & -\frac{1}{2}g_\Phi(\cdot \lrcorner *_\Phi T^5_{48},\cdot \lrcorner *_\Phi T^5_{48} ) \\
        &+\frac{7}{6}\pi_7(\mathrm{d}T^1_8) +\frac{7}{12}\pi_7(*_{\Phi}(T^1_8 \wedge T^5_{48})).
    \end{split}         
    \end{equation}
\end{thm}

    As with the earlier formula expressing the Ricci tensor in terms of the torsion tensor and its covariant derivatives, this formula does not \emph{look} particularly useful for explicit computations. However, we see that in the balanced ($T^1_8 = 0$) case, these formulae simplify quite significantly and, in particular, the terms $T \star T, \mathcal{L}_{T_8}g$ and $\operatorname{div}T$ vanish. In general, the balanced case need not be preserved along the flow\footnote{This can be seen by considering the evolution equation of the torsion tensor $T$ \cite[Theorem 3.4]{KarigiannisFlows}, and using that to obtain the evolution equation for $T_8$.}, but we will see an example in the following section that starts off balanced and remains balanced under the flow (Example \ref{ExampleOtherSolution}). So, the above formula provide a great reduction of the necessary calculations. Even in the non-balanced case, the above formulae are often easier to deal with than explicitly computing the full torsion tensor and its various covariant derivatives and traces.
\section{Explicit solutions on homogeneous manifolds}\label{SectionExamples}
In this section, we construct the first explicit solutions to the flows discussed so far. We obtain non-trivial solutions, including shrinking solitons, to the gradient flow and the coupled Ricci-harmonic flow, as well as a stationary solution to the harmonic flow. All of these examples are in the homogeneous setting, so we begin by reviewing the theory of homogeneous spaces and homogeneous \Sp-structures. 

\begin{defn}
    Let $G$ be a Lie group acting on a manifold $M$. The action of $G$ on $M$ is called \emph{almost effective} if the kernel of the action is finite. That is, the set
    \[
    \{g \in G \mid g(x) = x ~\forall x \in M\}
    \] is finite.
\end{defn}

\begin{defn}
    A Riemannian manifold $(M,g)$ is called a \emph{homogeneous space} if there exists a Lie group $G$ which acts transitively and {almost effectively} by isometries on $M$. If this is the case, $M$ can be expressed as $M=G/H$, where $H$ is defined to be the stabiliser of a fixed point $o \in M$ under the action of $G$ on $M$. The presentation of a homogeneous space as $G/H$ where $G$ and $H$ are minimal is called the \emph{canonical presentation}.
\end{defn}
Intuitively, this says that the manifold \q{looks the same} at every point, since for any $x,y \in M$, there exists $g\in G$ such that $g(x) = y$, and this map is an isometry. We will be interested in \Sp-structures on homogeneous spaces which are compatible with the homogeneous structure, so we make the following definition, from \cite{Homo8}.
\begin{defn}
    A \Sp-structure $\Phi$ on a homogeneous manifold $M = G/H$ is called \emph{homogeneous} or \emph{invariant} if the action of $G$ on $M$ preserves $\Phi$.
\end{defn}

A classification of the canonical presentations $G/H$ of all compact, simply-connected homogeneous $8$-manifolds admitting invariant \Sp-structures was given in \cite{Homo8}, and we recall this result here.

\begin{thm}[{\cite[Theorem B]{Homo8}}]\label{TheoremList}
    The canonical presentations of all compact, simply connected, almost effective homogeneous spaces admitting invariant Spin($7$)-structures are exhausted by:
    \begin{itemize}
        \item $\displaystyle \frac{\mathrm{SU}(3)}{\{e\}},$
        \item The infinite family 
        \[
        C_{k,l,m} \coloneqq \frac{\mathrm{SU}(2) \times \mathrm{SU}(2) \times \mathrm{SU}(2)}{\mathrm{U}(1)_{k,l,m}}, \quad \text{with } k = l + m,
        \]
        where $\mathrm{U}(1)_{k,l,m} \coloneq \{(z^k,z^l,z^m) \mid z\in \mathrm{U}(1)\},$
        \item The Calabi--Eckmann manifold $\displaystyle \frac{\mathrm{SU}(3)}{\mathrm{SU}(2)} \times \mathrm{SU}(2).$
    \end{itemize}
\end{thm}

We note that, by short-time uniqueness of solutions to the gradient flow \eqref{GF}, a homogeneous \Sp-structure remains homogeneous as it evolves under the flow.

With this in mind, we begin to investigate some explicit examples of \Sp-structures on homogeneous spaces, with the hope of describing how they evolve under the gradient flow \ref{GF}. The most natural choice of underlying space is $\mathrm{SU}(3)$. The torsion forms are not of a particularly simple form in this case, so our formulae from Section \ref{SectionReformulation} do not simplify things, so we instead compute the full torsion tensor $T$ using the method described in Section \ref{SectionPrelims}.

\begin{example}\label{ExampleSU3Flow} \footnote{We note that this example of a shrinking soliton \Sp-structure on $\operatorname{SU}(3)$ has also been independently discovered by Dwivedi-Singhal (Private communication, October 2025).}
    In \cite[Section 7]{Fernandez}, a \Sp-structure inducing the bi-invariant metric on $\operatorname{SU}(3)$ is written down, with respect to a particular choice of basis. With their choice of basis, this is written as
    \begin{equation}
    \begin{split}
         \Phi = &e^{1235}+e^{1248} + e^{1267}+e^{1346} +e^{1378}+e^{1457} + e^{1568}\\
        & +e^{2347} -e^{2368}-e^{2456} + e^{2578}+e^{3458}-e^{3567}-e^{4678},
    \end{split}
    \end{equation}
    inducing the metric $g = \operatorname{diag}(1, \cdots ,1).$
Note that this choice of \Sp-form differs from the one we used to define \Sp, but only by a change of basis. Explicitly, this change of basis is given by relabelling the element Fernandez calls $1$ as $e_1$, adding $2$ to each index in $e_0,\cdots ,e_6$ and then applying the change of basis given by
\[
\begin{cases}
    e_4\mapsto e_5,\\
    e_5 \mapsto e_4,\\
    e_6 \mapsto e_7,\\
    e_7 \mapsto -e_8,\\
    e_8 \mapsto e_6,\\
    e_i\mapsto e_i \text{ for all } i=1,2,3.    
\end{cases}
\]
For the sake of consistency with \cite{Fernandez}, we will work in the basis given in this example.    
    We then compute the torsion tensor $T$ using the method outlined in Section \ref{SectionPrelims}, and the aid of Maple for the symbolic algebraic manipulation. The intermediate calculations are unpleasant and not particularly enlightening, but we include all details in Appendix \ref{Appendix}, including an explicit expression of the torsion tensor.
    From this, we obtain the following components of the flow equation.
\begin{equation}
\begin{aligned}
    \operatorname{Ric} &= \operatorname{diag}\!\left(\tfrac{1}{4}\right), 
    &\quad T \star T &= \operatorname{diag}\!\left(\tfrac{3}{2}\right), \\[6pt]
    | T |^2 &= 2, 
    &\quad \mathcal{L}_{T_8} g &= \operatorname{div} T = 0.
\end{aligned}
\end{equation}
    We note that here, and for all explicit computations of flows, we represent the $2$-tensors involved as $8$-by-$8$ matrices in the usual way, using the basis $\{e^1,\cdots, e^8\}$.
    With this, writing $$A= (-\operatorname{Ric} + 2(\mathcal{L}_{T_8}g) + T\star T - |T|^2g + 2 \text{ div}T),$$  we compute $A = \operatorname{diag}\left(-\frac{3}{4}\right) = -\frac{3}{4}g$, which implies that $A \diamond \Phi = -3\Phi$ and so $\Phi$ satisfies the soliton equation (\ref{EquationSoliton}) and is a shrinker. We can explicitly solve the flow, by taking the ansatz $\Phi_t = f(t)^4\Phi$ with $f\in C^\infty(\mathbb{R})$ such that $f(0) = 1$. Note that this rescaled \Sp-structure induces the rescaled metric $g_t = \operatorname{diag}(f(t)^2, \cdots, f(t)^2)$.
    Substituting this into the flow equation and using the scaling of the right hand side by \eqref{EquationRescaledFlow} gives the following ODE for $f$:
    \begin{equation}
        4f(t)^3\frac{df(t)}{dt} = -3f(t)^2,
    \end{equation}
    which we solve to find $f(t) = \frac{1}{\sqrt2}\sqrt{2-3t}$ which finally gives 
    \begin{equation}
        \Phi_t = \frac{1}{4}(2-3t)^2\Phi_0
    \end{equation}
    as the solution to the flow equation \ref{GF} starting from the initial \Sp-structure $\Phi$ on $\operatorname{SU}(3)$.
\end{example}

    The calculations above also give rise to solutions for the other two flows of \Sp-structures mentioned in Section \ref{SectionFlows}. Since $\operatorname{div}T = 0$ and $\operatorname{Ric} = \operatorname{diag}\left(\frac{1}{4}\right)$, the \Sp-structure $\Phi$ defined on $\operatorname{SU}(3)$ above is a critical point for the harmonic flow and a shrinking soliton for the coupled Ricci-harmonic flow. To the best of our knowledge, these are the first explicit (non-stationary) solutions to any flow of \Sp-structures. The example above proves, in particular, that compact shrinking solitons exist for the gradient flow of \Sp-structures, in contrast to compact expanders. We collect all of this into the following theorem.

\begin{thm}\label{ThmSolitonsExist}
    There exist compact shrinking solitons to the gradient flow of \Sp-structures given by \eqref{GF}. In particular, 
    \begin{equation}
        \Phi_t = \frac{1}{4}(2-3t)^2\Phi_0
    \end{equation}
    is a shrinking soliton solution to the gradient flow \eqref{GF} starting from the initial \Sp-structure on $\operatorname{SU}(3)$, described in Example \ref{ExampleSU3Flow}. Moreover, the one-parameter family of \Sp-structures given by 
    \[
     \Phi_t = \frac{1}{4}(2-t)^2\Phi_0
    \]
    is a shrinking soliton for the Ricci-harmonic flow \eqref{EquationRicciHarmonic}, with the same initial conditions.
\end{thm}
We now proceed with a non-simply-connected example, which is therefore not given in the classification of Theorem \ref{TheoremList}, but has appeared in the literature (\cite{Niedzialomski},\cite{Cabrera}). In this case, the torsion forms take a very simple form, so the formulae from Section \ref{SectionReformulation} greatly simplifiy the calculations.

\begin{example}\label{ExampleOtherSolution}
    We begin by describing the construction of the underlying manifold. This manifold and its \Sp-structure were originally discussed in \cite{Cabrera}, and considered more recently in \cite[Example 4.7]{Niedzialomski}.
  
Let \( k \in \mathbb{R} \setminus \{0\} \). Define the Lie group \( G(k) \) as the set of matrices of the form:

\[
G(k) = \left\{ 
\begin{pmatrix}
e^{kz} & 0 & 0 & x \\
0 & e^{-kz} & 0 & y \\
0 & 0 & 1 & z \\
0 & 0 & 0 & 1 \\
\end{pmatrix}
: x, y, z \in \mathbb{R}
\right\}.
\]
Note that $G(k)$ is isomorphic to the $3$-dimensional Lie group $\mathbf{Sol}$ \cite{Sol} and is non-compact, and $\{x,y,z\}$ is a coordinate system on $G(k)$.
We have a basis of right-invariant $1$-forms on $G(k)$ given by
$$\{\eta_1 = \mathrm{d}x-kx\mathrm{d}z, \eta_2 = \mathrm{d}y+ky\mathrm{d}z, \eta_3 = \mathrm{d}z\}.$$ There exists a discrete subgroup $\Gamma(k)$ such that $G(k)/\Gamma(k)$ is a compact manifold \cite{Sol}. 
Since the basis of $1$-forms is right-invariant, this basis descends to a basis on the compact quotient manifold $$H(k) = G(k)/\Gamma(k).$$
We note that $H(k)$ is a $T^2$-bundle over $S^1$ (since $\mathbf{Sol} \cong \mathbb{R}^2 \rtimes \mathbb{R}$), with the $z$-coordinate spanning the $S^1$-factor and the $\{x,y\}$-coordinates spanning the torus. 
Finally, define 
\[M = H(k)\times T^5.\]
Let $\{e^1,e^2,e^3,e^4,e^5\}$ be a basis of closed $1$-forms on $T^5$. Then,
\[\{e^1,e^2,e^3,e^4,e^5,e^6 = \eta_1, e^7 = \eta_2, e^8 = \eta_3\}\]is a basis of $1$-forms on $M$. Define a metric $g$ on $M$ such that the above basis of $1$-forms is orthonormal. From the definitions of $\eta_1$ and $\eta_2$, and the fact that $e^1,\cdots,e^5$ are closed, we see that
\begin{equation}
    \begin{split}
        \mathrm{d}e^i &= 0, \text{for all }i=1,2,3,4,5,8,\\
        \mathrm{d}e^6 &= \mathrm{d}(\mathrm{d}x-kx\mathrm{d}z) = -k\mathrm{d}x \wedge \mathrm{d}z = -ke^{67},\\
        \mathrm{d}e^7 &= \mathrm{d}(\mathrm{d}y+ky\mathrm{d}z) = k\mathrm{d}y\wedge \mathrm{d}z = ke^{78}.
    \end{split}
\end{equation}
Note that in this case there are only two non-closed $1$-forms. This makes the computation of the gradient flow equation much easier than it was in Example \ref{ExampleSU3Flow}, so we will not have to relegate it to an appendix.
From now on, we will take $k$ to be $1$, and at the end of the construction we will remark on what changes for any other choice of non-zero $k$.

On this manifold, we can write down the following \Sp-structure:
\begin{equation}
    \begin{split}
        \Phi = &e^{1234}-e^{1256} - e^{1278}-e^{1357} +e^{1368}-e^{1458} - e^{1467}\\
        & +e^{5678} -e^{3478}-e^{3456} - e^{2468}+e^{2457}-e^{2367}-e^{2358},
    \end{split}
\end{equation}
which induces the Riemannian metric for which the basis defined above is orthonormal.
Note again that this form differs from the one we used to define \Sp. Again, this is only by a change of basis, this time given by $e_5 \mapsto -e_5$ and $e_8 \mapsto -e_8$.
Using the exterior derivatives of the basis $1$-forms, we compute
\begin{equation}
\begin{split}
    \mathrm{d}\Phi &= - e^{12568}+e^{13578}-e^{34568}-e^{24578},\\
    *\mathrm{d}\Phi &= e^{347}-e^{246}+e^{127}-e^{136},
\end{split} 
\end{equation}
at which point we see $*\mathrm{d}\Phi \wedge \Phi = 0$ and so $T^1_8=0$ and $T^5_{48} = \mathrm{d}\Phi$. This simplifies the formula for the evolution equation obtained in Section \ref{SectionReformulation}, and we compute the following:
\begin{equation}
    \begin{split}
        \operatorname{Ric} &= \operatorname{diag}(0,0,0,0,0,0,0,-2),\\
        T\star T&= \mathcal{L}_{T_8}g = \operatorname{div}T=0,\\
       | T|^2&=1.
    \end{split}
\end{equation}
Letting $A= (-\operatorname{Ric} + 2(\mathcal{L}_{T_8}g) + T\star T - |T|^2g + 2 \text{ div}T)$, we compute that
\begin{equation}
    \begin{split}
        A \diamond\Phi &= -4(e^{1234}-e^{1256}-e^{1357}- e^{1467}-e^{3456}+e^{2457}-e^{2367})\\ & -2(- e^{1278}+e^{1368}-e^{1458}+e^{5678} -e^{3478}- e^{2468}-e^{2358}).
    \end{split}
\end{equation}
Observing this initial behaviour under the flow, we see that all of the terms involving $e^8$ shrink with initial speed $2$ and all other terms shrink with initial speed $4$, so a natural choice of ansatz is to write $e_8 \mapsto h(t)e_8$ and $e_i \mapsto f(t)e_i$ for all $i\neq 8$, for some $f,h\in C^\infty(\mathbb{R})$ with $f(0)=h(0) = 1$. This gives
\begin{equation}
\begin{split}
    \Phi_t = &f(t)^4(e^{1234}-e^{1256}-e^{1357}- e^{1467}-e^{3456}+e^{2457}-e^{2367})\\
    &+f(t)^3h(t)(- e^{1278}+e^{1368}-e^{1458}+e^{5678} -e^{3478}- e^{2468}-e^{2358}),
\end{split} 
\end{equation}
which induces the metric $g_t = \operatorname{diag}(f(t)^2,\cdots, f(t)^2,h(t))$.
Substituting this ansatz into the gradient flow equation \ref{GF}, we obtain
\begin{equation}
\begin{split}
    \frac{\partial\Phi_t}{\partial t} = &-\frac{4f(t)^4}{h(t)^2}(e^{1234}-e^{1256}-e^{1357}- e^{1467}-e^{3456}+e^{2457}-e^{2367})\\
    &-\frac{2f(t)^3}{h(t)}(- e^{1278}+e^{1368}-e^{1458}+e^{5678} -e^{3478}- e^{2468}-e^{2358}).
\end{split} 
\end{equation}
Equating basis $4$-forms yields the following system of coupled ODEs.
\[
\begin{cases}
\displaystyle 4f(t)^3 \frac{df(t)}{dt} = -\frac{4f(t)^4}{h(t)^2}, \\
\displaystyle f(t)^3 \frac{dh(t)}{dt} + 3f(t)^2 \frac{df(t)}{dt} h(t) = -\frac{2f(t)^3}{h(t)}, \\
f(0) = 1, \quad h(0) = 1,
\end{cases}
\]
which can easily be decoupled and solved to see that
\[
\begin{cases}
f(t)= \displaystyle \frac{1}{(2t+1)^\frac{1}{2}},\\
h(t)=(2t+1)^\frac{1}{2},
\end{cases}
\]
so the solution to the gradient flow with initial \Sp-structure given by $\Phi$ is 
\begin{equation}
\begin{split}
    \Phi_t = &\frac{1}{(2t+1)^2}(e^{1234}-e^{1256}-e^{1357}- e^{1467}-e^{3456}+e^{2457}-e^{2367})\\
    &+\frac{1}{(2t+1)}(- e^{1278}+e^{1368}-e^{1458}+e^{5678} -e^{3478}- e^{2468}-e^{2358}).
\end{split} 
\end{equation}
Now that we have an explicit solution to the gradient flow, we can explicitly describe the evolution of the underlying metric. The metric induced by $\Phi_t$ is
\begin{equation}
    g_t = \operatorname{diag}\left(\frac{1}{2t+1},\cdots,\frac{1}{2t+1},2t+1 \right),
\end{equation}
so the underlying manifold $M$ evolves in such a way that the $7$-dimensional fibres over $S^1$ shrink and the base $S^1$ expands linearly in $t$. This behaviour is consistent with other geometric flows on bundles over circles.

Finally, for $k$ not equal to $1$, all that changes is $\mathrm{d}\Phi$, and therefore $T$, picks up an overall factor of $k$, which implies that $A$ is scaled by a factor of $k^2$ and so the ODEs for $f$ and $h$ each have a factor of $k^2$ on the right hand side, leading to the solution
\begin{equation}
\begin{split}
    \Phi_t = &\frac{k^2}{(2t+1)^2}(e^{1234}-e^{1256}-e^{1357}- e^{1467}-e^{3456}+e^{2457}-e^{2367})\\
    &+\frac{k^2}{(2t+1)}(- e^{1278}+e^{1368}-e^{1458}+e^{5678} -e^{3478}- e^{2468}-e^{2358}).
\end{split} 
\end{equation}
\end{example}

For the spaces $C_{k,l,m}$ and the Calabi-Eckmann manifold mentioned in Theorem \ref{TheoremList}, the authors of \cite{Homo8} construct $5$-parameter families of homogeneous \Sp-structures. As work in progress, we are currently studying the gradient flow on these families, with the aim of understanding the dynamical system of the evolution of a general \Sp-structure. In particular, it would be interesting to see if a general \Sp-structure on these spaces always converges to a particular one, modulo rescaling. This would suggest that the limiting \Sp-structure was somehow canonical.

\subsection{Linear instability of solitons}\label{SectionStability}

Now that we know shrinking solitons exist for the gradient flow of \Sp-structures, it is natural to consider the question of if and when they are \emph{stable} critical points for the (renormalised) flow.  From the spinorial point of view, Schiemanowski proved that all solitons are stable critical points for the renormalised spinor flow, so long as the associated metric admits no Killing fields \cite[Theorem 2]{schiemanowski2017stabilityspinorflow}. In all of the examples considered so far, there \emph{are} Killing fields so there is more work to be done. Indeed, we shall see that the soliton constructed on $\operatorname{SU}(3)$ (Example \ref{ExampleSU3Flow}) is in fact an unstable critical point.

For a soliton $\Phi_t$ evolving under the gradient flow \ref{GF} as $\frac{\partial}{\partial t}\Phi_t = \lambda \Phi_t$, we consider the renormalised flow 
\begin{equation}
    \frac{\partial}{\partial t} \Phi(t) = (-\operatorname{Ric} + 2(\mathcal{L}_{T_8}g) + T\star T - |T|^2g + 2 \text{ div}T)_t \diamond_t \Phi(t)- \lambda \Phi(t),
\end{equation}
so that the soliton $\Phi_t$ is a critical point for the flow. We will call this flow the \emph{renormalised gradient flow} of \Sp-structures, for a soliton with scaling factor $\lambda$. Since the soliton $\Phi(t)$ is a critical point for the renormalised flow, we can examine the stability of the flow around this point. Since we have an explicit example of a shrinking soliton on $\operatorname{SU}(3)$ (Example \ref{ExampleSU3Flow}) we can examine its stability, and we obtain the following result.
\begin{thm}\label{ExampleSU3Unstable}
        The \Sp-soliton constructed in Example \ref{ExampleSU3Flow} is an unstable critical point for the renormalised gradient flow \eqref{EquationRescaledFlow}, modulo rescaling. Moreover, there exist stable directions, unstable directions and zero modes.
    \end{thm}

\begin{proof}
    Recall that the soliton $\Phi_t$ constructed in Example \ref{ExampleSU3Flow} satisfies $A_0 \diamond_0 \Phi_0 = -3\Phi_0$, where $A = (-\operatorname{Ric} + 2(\mathcal{L}_{T_8}g) + T\star T - |T|^2g + 2 \text{ div}T)$ is induced by $\Phi_0$.
    So, we consider the flow
    \begin{equation}
    \frac{\partial}{\partial t} \Phi(t) = (-\operatorname{Ric} + 2(\mathcal{L}_{T_8}g) + T\star T - |T|^2g + 2 \text{ div}T)_t \diamond_t \Phi(t) +3 \Phi(t),
\end{equation}
    for which $\Phi_0$ is a critical point, so starting this renormalised flow at $\Phi_0$ gives a stationary solution $\Phi_{0,t}$. To examine stability, we take a variation of $\Phi_0$ within the space of \Sp-structures, and study how it evolves under this flow. With this in mind, we let 
   \begin{equation}
    \begin{split}
         \Phi_{s,0} = &e^se^{1235}+e^{1248} + e^{1267}+e^{1346} +e^{1378}+e^{1457} + e^{1568}\\
        & +e^{2347} -e^{2368}-e^{2456} + e^{2578}+e^{3458}-e^{3567}-e^{-s}e^{4678},
    \end{split}
    \end{equation}
    where the terms $e^s$ and $e^{-s}$ refer to the exponential function.
    This family of \Sp-structures induces the metric
    \begin{equation}
        g_{s,0} = \operatorname{diag}(e^{s/2},e^{s/2},e^{s/2},e^{-s/2},e^{s/2},e^{-s/2},e^{-s/2},e^{-s/2}).
    \end{equation}
    We define 
    \begin{equation}
        \Psi_t = \frac{\partial \Phi_{s,t}}{\partial s}\Bigr \vert_{s=0},
    \end{equation}
    and note that 
    \[
    \Psi_0 = e^{1235}+e^{4678},
    \]
  
    which satisfies $* \Psi_0 = - \Psi_0$, so $\Phi_s$ is a deformation of $\Phi_0$ in one of the $\Omega^4_{35}$-directions (see \eqref{forms}). Note that this deformation is orthogonal to the direction $\Omega^4_1$, which corresponds to rescalings of $\Phi$.
   To study stability, we consider the evolution $\frac{\partial}{\partial t}|_{t=0} \langle\Psi_t,\Psi_t \rangle$, where $\langle \cdot,\cdot\rangle$ is the inner product on $4$-forms induced by $g_{\Phi_{0,t}}=g_{\Phi_{0,0}}$. If this quantity is negative, then, to first order, $\Phi_{s,t}$ is flowing back towards $\Phi_0$, and if it is positive then $\Phi_{s,t}$ flows away from $\Phi_0$, to first order.
   We see that
   \begin{equation}
   \begin{split}
       \frac{\partial}{\partial t}\Bigr \vert_{t=0} \langle\Psi_t,\Psi_t \rangle &= 2\left \langle \frac{\partial}{\partial t}\Bigr \vert_{t=0} \Psi_t,\Psi_0 \right \rangle\\
       &=2\left \langle \frac{\partial}{\partial t}\Bigr \vert_{t=0}\left( \frac{\partial \Phi_{s,t}}{\partial s}\Bigr \vert_{s=0} \right) ,\Psi_0 \right \rangle\\
       &= 2\left \langle \frac{\partial}{\partial s}\Bigr \vert_{s=0}\left( \frac{\partial \Phi_{s,t}}{\partial t}\Bigr \vert_{t=0} \right) ,\Psi_0 \right \rangle\\
       &= 2\left \langle \frac{\partial}{\partial s}\Bigr \vert_{s=0}\left( A_{s,0} \diamond_{s,0} \Phi_{s,0} + 3\Phi_{s,0} \right) ,\Psi_0 \right \rangle.\\
   \end{split}
    \end{equation}
    Finally, computing $\frac{\partial}{\partial s}\Bigr \vert_{s=0}\left( A_{s,0} \diamond_{s,0} \Phi_{s,0}  + 3\Phi_{s,0}\right)$ and taking the inner product results in \[\left \langle \frac{\partial}{\partial s}\Bigr \vert_{s=0}\left( A_{s,0} \diamond_{s,0} \Phi_{s,0} + 3\Phi_{s,0}  \right) ,\Psi_0 \right \rangle =  \frac{67}{18},\] which is positive.
    Thus, the soliton $\Phi_0$ is an unstable critical point for the renormalised gradient flow.

    Here, we have explicitly shown instability by showing that the deformation in the direction of the particular $\Psi_0$ above is an unstable deformation. We can repeat the same argument, starting with a deformation in six other anti-self-dual directions, by defining $\Phi_s$ as above but with the coefficients $e^s$ and $e^{-s}$ on each of the six other Hodge dual pairs of $4$-forms appearing in the expression of $\Phi_{0,0}$. For instance, to obtain a deformation in the direction of $e^{1248}+e^{3567} \in \Omega^4_{35}$, we can take
    \begin{equation}
    \begin{split}
         \Phi_{s,0} = &e^{1235}+e^se^{1248} + e^{1267}+e^{1346} +e^{1378}+e^{1457} + e^{1568}\\
        & +e^{2347} -e^{2368}-e^{2456} + e^{2578}+e^{3458}-e^{-s}e^{3567}-e^{4678}.
    \end{split}
    \end{equation}
    
    A computation gives that all seven of these directions are unstable.

    For deformations in the directions of the other $28$ anti-self-dual $4$-forms, we can proceed as follows. Take $\Phi_{0,0}$ as before and consider a deformation in the direction of the anti-self-dual $4$-form \[
    \Psi_0 = e^{1234} - e^{5678}.
    \]
    To first order\footnote{Note that this $\Phi_{s,0}$ does not define a \Sp-structure for all $s$, in contrast to our explicit deformation in the directions considered above. However, for stability, we are only interested in the first order of the deformation since we will only deal with $\frac{\partial}{\partial s}|_{s=0}\Phi_s$.}, such a deformation can be written as
    \begin{equation} 
        \Phi_{s,0} = \Phi_{0,0} + s\Psi_0,
    \end{equation}
    and the induced metric can be written as 
    \begin{equation}
        g_{s,0} = g_{0,0} + s(-e^1\circ e^7 + e^2 \circ e^6 -e^3 \circ e^8 +e^4\circ e^5),
    \end{equation}
    where $e^i \circ e^j = \frac{1}{2}(e^i \otimes e^j + e^j \otimes e^i)$ denotes the symmetric product of tensors.
    Now, starting from the initial \Sp-structure $\Phi_{s,0}$ and running the renormalised flow as above, a very similar computation gives that
    \begin{equation}
        \frac{\partial}{\partial t}\Bigr \vert_{t=0} \langle\Psi_t,\Psi_t \rangle = -\frac{16}{9},
    \end{equation}
    showing that the direction corresponding to $\Psi_0$ is a \emph{stable} direction. 
    
    The explicit deformations we have considered so far show that anti-self-dual deformations can correspond to both stable and unstable directions. The other subspace of $4$-forms in which we could deform $\Phi_0$ is the space $\Omega^4_7$, which corresponds to those deformations which do not affect the induced Riemannian metric. Because of this, one might expect that these arise as zero modes (i.e. directions that are neither stable nor unstable) and we now show that this is indeed the case.

    Given a \Sp-structure $\Phi_0$, an explicit family of \Sp-structures $\Phi_s$ corresponding to infinitesimal deformations in the $\Omega^4_7$-direction for each $s$ was given in \cite[Theorem 5.3.3]{KarigiannisDeformations}. In our case, with $\Phi_0$ as above, this becomes
    \begin{equation}
        \begin{split}
            \Phi_s = \Phi_0 &+(1- \cos(s))(-e^{2347} + e^{2368} + e^{2456} - e^{2578} - e^{1346} + e^{1378} + e^{1457} + e^{1568})\\
                            &+\sin(s)(e^{2346} + e^{2378} + e^{2457} + e^{2568} - e^{1347} - e^{1368} - e^{1456} + e^{1578}),
        \end{split}
    \end{equation}
    and, as mentioned above, the metric is not affected, so $g_s = g_0$ for all $s$.    
    Once again, repeating the computation above gives 
    \begin{equation}
        \frac{\partial}{\partial t}\Bigr \vert_{t=0} \langle\Psi_t,\Psi_t \rangle = 0,
    \end{equation}
    so this direction of deformations is a zero mode. 
    \end{proof}

    \begin{remark}
        We briefly discuss why we may have expected the deformations in the $\Omega^4_{35}$ to exhibit both stable and unstable directions. A priori, $\Lambda^4_{35}$ is isomorphic to an irreducible \Sp-representation, so we may expect there to be no further splitting, and for all $35$ directions to behave in the same way. However, since our underlying manifold is $\operatorname{SU(3)}$, we also have the presence of a natural action of $\operatorname{SU}(3)$. Recall that $\Lambda^4_{35} \cong S^2_0(\mathbb{R}^8)$. As an $\operatorname{SU}(3)$-representation, we have
        \[
        S^2_0(\mathbb{R}^8) \cong \boldsymbol{8} \oplus \boldsymbol{27},
        \]
        so there is a further decomposition that we are observing here.
    \end{remark}
    We note that this result is not necessarily unexpected. Stable solitons of geometric flows are generally quite rare; for example, a compact Ricci soliton is stable if and only if it is a local maximum for Perelman's shrinker entropy \cite{RicciSoliton}. Additionally, this instability could prove to be a desirable property in using the gradient flow to find torsion-free \Sp-structures. Indeed, one way for the (non-renormalised) gradient flow to develop a singularity would be for it to first flow to a shrinking soliton, and then shrink to a point. The instability illustrated in this example shows that this is unlikely to happen for this particular soliton, and if it did, then small perturbations of the initial condition may avoid the shrinking to a singularity.

\appendix
\section{Computational details for the gradient flow on \texorpdfstring{$\operatorname{SU}(3)$}{SU(3)}}\label{Appendix}
Here, we collect the computational details used in the derivation of the flow equation for the \Sp-structure on $\operatorname{SU}(3)$ discussed in Example \ref{ExampleSU3Flow} and the proof of Theorem \ref{ExampleSU3Unstable}. We work with the following basis of $\mathfrak{su}(3) = T_e\operatorname{SU}(3)$, as used in \cite{Fernandez}.

\[
\begin{aligned}
e_1 &= \frac{1}{\sqrt{12}}
\begin{pmatrix}
\mathrm{i} & 0 & 0\\
0 & 0 & 0\\
0 & 0 & -\mathrm{i}
\end{pmatrix},\qquad
e_2 = \frac{1}{\sqrt{12}}
\begin{pmatrix}
0 & 1 & 0\\
-1 & 0 & 0\\
0 & 0 & 0
\end{pmatrix},\qquad
e_3 = \frac{1}{\sqrt{12}}
\begin{pmatrix}
0 & 0 & 1\\
0 & 0 & 0\\
-1 & 0 & 0
\end{pmatrix},\\[8pt]
e_4 &= \frac{1}{\sqrt{12}}
\begin{pmatrix}
0 & 0 & 0\\
0 & 0 & 1\\
0 & -1 & 0
\end{pmatrix},\qquad
e_5 = \frac{1}{6}
\begin{pmatrix}
\mathrm{i} & 0 & 0\\
0 & -2\mathrm{i} & 0\\
0 & 0 & \mathrm{i}
\end{pmatrix},\qquad
e_6 = \frac{1}{\sqrt{12}}
\begin{pmatrix}
0 & \mathrm{i} & 0\\
\mathrm{i} & 0 & 0\\
0 & 0 & 0
\end{pmatrix},\\[8pt]
e_7 &= \frac{1}{\sqrt{12}}
\begin{pmatrix}
0 & 0 & \mathrm{i}\\
0 & 0 & 0\\
\mathrm{i} & 0 & 0
\end{pmatrix},\qquad
e_8 = \frac{1}{\sqrt{12}}
\begin{pmatrix}
0 & 0 & 0\\
0 & 0 & \mathrm{i}\\
0 & \mathrm{i} & 0
\end{pmatrix}.
\end{aligned}
\]
From the Koszul formula, we have that for all $v,w \in \mathfrak{su}(3) $,
\begin{equation}
    \mathrm{d}e^i(v,w) = -e^i([v,w]),
\end{equation}
where $[v,w]$ is the commutator of matrices. We obtain
\begin{align*}
\mathrm{d}e^1 &= -\frac{\sqrt{3}}{6} e^{26}
       -\frac{\sqrt{3}}{3} e^{37}
       -\frac{\sqrt{3}}{6} e^{48}, \\
\mathrm{d}e^2 &= \frac{\sqrt{3}}{6} e^{16}
       +\frac{\sqrt{3}}{6} e^{34}
       +\frac12 e^{56}
       +\frac{\sqrt{3}}{6} e^{78}, \\
\mathrm{d}e^3 &= \frac{\sqrt{3}}{3} e^{17}
       -\frac{\sqrt{3}}{6} e^{24}
       +\frac{\sqrt{3}}{6} e^{68}, \\
\mathrm{d}e^4 &= \frac{\sqrt{3}}{6} e^{18}
       +\frac{\sqrt{3}}{6} e^{23}
       -\frac12 e^{58}
       +\frac{\sqrt{3}}{6} e^{67}, \\
\mathrm{d}e^5 &= -\frac12 e^{26}
       +\frac12 e^{48}, \\
\mathrm{d}e^6 &= -\frac{\sqrt{3}}{6} e^{12}
       +\frac12 e^{25}
       -\frac{\sqrt{3}}{6} e^{38}
       -\frac{\sqrt{3}}{6} e^{47}, \\
\mathrm{d}e^7 &= -\frac{\sqrt{3}}{3} e^{13}
       -\frac{\sqrt{3}}{6} e^{28}
       +\frac{\sqrt{3}}{6} e^{46}, \\
\mathrm{d}e^8 &= -\frac{\sqrt{3}}{6} e^{14}
       +\frac{\sqrt{3}}{6} e^{27}
       +\frac{\sqrt{3}}{6} e^{36}
       -\frac12 e^{45}.
\end{align*}
Using this, we compute
\begin{align*}
\mathrm{d}\Phi &= -\frac12 e^{12345} - \frac12 e^{12348} + \frac{\sqrt{3}}{3} e^{12358} 
           - \frac{\sqrt{3}}{6} e^{12457} + \frac{\sqrt{3}}{6} e^{12458} + \frac12 e^{12467} 
           - \frac{\sqrt{3}}{6} e^{12567} + \frac{\sqrt{3}}{2} e^{12678} \\
&\quad + \frac12 e^{13457} + \frac{\sqrt{3}}{2} e^{13468} + \frac12 e^{13568} 
           - \frac{\sqrt{3}}{6} e^{13578} - \frac12 e^{14568} + \frac{\sqrt{3}}{2} e^{14578} 
           - \frac{\sqrt{3}}{6} e^{15678} - \frac12 e^{23456} \\
&\quad + \frac{\sqrt{3}}{6} e^{23457} - \frac12 e^{23468} + \frac{\sqrt{3}}{2} e^{23478} 
           + \frac{\sqrt{3}}{6} e^{23567} + \frac{\sqrt{3}}{6} e^{23568} + \frac12 e^{23578} 
           - \frac{\sqrt{3}}{6} e^{24568} - \frac12 e^{24578} \\
&\quad + \frac12 e^{34567} + \frac{\sqrt{3}}{6} e^{34578} + \frac12 e^{34678} - \frac{\sqrt{3}}{2} e^{35678} 
           - \frac{\sqrt{3}}{6} e^{45678},\\
T^1_8 &= \frac{1}{7}\left(-3e^3+3e^4-\sqrt{3}e^5 +3\sqrt{3}e^8\right),
\end{align*}
and using the method discussed in Section \ref{SectionPrelims}, we obtain 
\begin{align*}
   T &= \frac{\sqrt{3}}{48} e^1 \otimes\left(
e^{12} + e^{17} + 2 e^{18}
- 2 e^{24} + e^{26}
- e^{35} + 2 e^{37} + e^{38}
- e^{45} + e^{48}
- 2 e^{56} - e^{67}
\right)\\
& +e^2 \otimes   \bigg(-\frac{\sqrt{3}}{48} e^{13} - \frac{\sqrt{3}}{24} e^{16} - \frac{1}{16} e^{18} + \frac{1}{16} e^{24} - \frac{\sqrt{3}}{48} e^{25} - \frac{\sqrt{3}}{24} e^{34} + \frac{1}{16} e^{37} + \frac{\sqrt{3}}{48} e^{46}\\
&- \frac{1}{16} e^{56} - \frac{\sqrt{3}}{48} e^{78}  \bigg )\\
 &+  \frac{\sqrt{3}}{48} e^3 \otimes\left(
- e^{15} - 2 e^{17} + e^{18} + e^{23} + e^{24}
- 2 e^{26} - e^{37} + 2 e^{38}
- 2 e^{45} - e^{47} - e^{56} - e^{68}
\right)\\
&+  \frac{\sqrt{3}}{48}e^4 \otimes (- e^{12}- e^{15}- \sqrt{3} e^{16}- e^{18}- e^{23}- e^{24}- \sqrt{3} e^{27}+ \sqrt{3} e^{34} \\
&\hspace{1cm}+e^{35}- e^{37}- e^{47}+ e^{48}- e^{56}+ \sqrt{3} e^{58}- e^{67}+ e^{68})\\
 &+\frac{1}{16}e^5\otimes(- e^{12}+ e^{17}+ e^{26}+ e^{35}+ e^{38}- e^{45}- e^{48}+ e^{67})\\
&+\frac{\sqrt{3}}{48}e^6 \otimes (e^{12}+\sqrt{3}e^{13}-e^{15}-e^{17}+e^{23}-\sqrt{3}e^{25}+e^{26}+e^{35}\\
&\hspace{1cm}+e^{38}+e^{45}-\sqrt{3}e^{46}+e^{47}+e^{48}+e^{67}+e^{68}-\sqrt{3}e^{78}),\\
 &+ \frac{\sqrt{3}}{48}e^7 \otimes (e^{13}-e^{14}-3e^{25}+e^{28}-e^{36}+e^{46}+e^{57}+3e^{78})\\
&+\frac{\sqrt{3}}{48}e^8 \otimes(2e^{14}+e^{16}+\sqrt{3}e^{17}-\sqrt{3}e^{26}-e^{27}-e^{34}-2e^{36}+\sqrt{3}e^{38}+\sqrt{3}e^{45}+e^{58}).\\
\end{align*}

Fortunately, the expressions for the terms of the evolution equation for the gradient flow \ref{GF} simplify considerably, and we list them in Example \ref{ExampleSU3Flow}.

\printbibliography

\end{document}